%% file: main_paper.tex
\documentclass[11pt]{article}
\usepackage[margin = 1.0 in]{geometry}

\usepackage{shay}
\usepackage{xcolor}

\usepackage{lineno}

\title{Comparison of Matrix Norm Sparsification%
  \thanks{Work partially supported by ONR Award N00014-18-1-2364,
  the Israel Science Foundation grant \#1086/18,
  the Israeli Council for Higher Education (CHE) via the Weizmann Data Science Research Center,
  and a Minerva Foundation grant.  
  }
}
    
\author{
  Robert Krauthgamer  
  \qquad
  Shay Sapir
  \\ Weizmann Institute of Science
  \\ \texttt{\{robert.krauthgamer,shay.sapir\}@weizmann.ac.il}  
  }
\date{}

\begin{document}

\maketitle

\begin{abstract}

    A well-known approach in the design of efficient algorithms, called matrix sparsification, approximates a matrix $A$ with a sparse matrix $A'$. Achlioptas and McSherry [2007] initiated a long line of work on spectral-norm sparsification, which aims to guarantee that $\|A'-A\|\leq \epsilon \|A\|$ for error parameter $\epsilon>0$. Various forms of matrix approximation motivate considering this problem with a guarantee according to the Schatten $p$-norm for general $p$, which includes the spectral norm as the special case $p=\infty$.
    
    We investigate the relation between fixed but different $p\neq q$, that is, whether sparsification in the Schatten $p$-norm implies (existentially and/or algorithmically) sparsification in the Schatten $q\text{-norm}$ with similar sparsity. An affirmative answer could be tremendously useful, as it will identify which value of $p$ to focus on. Our main finding is a surprising contrast between this question and the analogous case of $\ell_p$-norm sparsification for vectors: For vectors, the answer is affirmative for $p<q$ and negative for $p>q$, but for matrices we answer negatively for almost all sufficiently distinct $p\neq q$. In addition, our explicit constructions may be of independent interest.
\end{abstract}

\input{Introduction}

\input{EllSparsif}
\input{SchattenExamples}

\input{SmallPQ}

\input{appendix}

\bibliographystyle{alphaurl}

\bibliography{references.bib}

\end{document}

%% file: Introduction.tex
\section{Introduction}

Large graphs and matrices are ubiquitous in modern computations.
If the graphs and matrices are sparse,
then many computations are more efficient in terms of time, space and/or communication resources.
A well-known approach for leveraging this efficiency is sparsification,
where an input object (graph, vector, matrix, etc.)
is approximated using a sparse object.
This idea is used in a broad range of domains,
from combinatorics (e.g., graph spanners) 
through signal processing (e.g., sampling)
and statistics (e.g., Principal Component Analysis),
to numerical linear algebra (matrix sparsification).
We study this last domain, 
which focuses on the matrix spectrum, as formalized next. 

In matrix sparsification, the goal is to approximate a matrix $A$ using a matrix $A'$ that is sparse.
However, the notion of approximating a matrix, or even a vector, is not straightforward.
For vectors, a commonly used measure of approximation is the $\ell_p$-norm for $p\geq 1$, defined for a vector $x\in\R^n$ as
$\|x\|_p \coloneqq (\sum_{i=1}^n |x_i|^p)^{1/p}$.
Its matrix analogue is the Schatten $p$-norm, defined for a matrix $A\in\R^{n\times d}$ as
$\|A\|_{S_p} \coloneqq (\sum_{i=1}^r \sigma_i^p)^{1/p}$,
where $\sigma_1,...,\sigma_r$ are the singular values of $A$. 
These definitions extend to $p=\infty$ by taking the limit. 
They also extend to $p<1$, although they do not yield a norm.
Notice that $\ell_p$-norm is a special case of Schatten $p$-norm, restricted to diagonal matrices.
Special cases of the Schatten $p$-norm include $p=1,2,\infty$, which are the trace norm, Frobenius norm and spectral norm, respectively.
Another special case is the Schatten $0$-norm, defined as the rank of $A$.

The Schatten norms capture fundamental properties of the matrix.
In particular, multiple Schatten norms can be used for various forms of matrix approximation,
for instance, with respect to spectral sum functions~\cite{khetan2019spectrum}, or even the entire spectrum~\cite{kong2017spectrum}.
Needless to say that the Schatten $p$-norms for $p=0,2,\infty$ are of utmost importance, and in general, specific $p$ may be important for concrete applications.
For instance, for even $p$, the $p$-th power of the Schatten $p$-norm of a graph adjacency matrix is the number of closed walks of length $p$.
Another widely known application is of
the trace norm ($p=1$) as a relaxation for the rank \cite{DBLP:journals/cacm/CandesR12,DBLP:journals/tit/CandesT10}, and even other values of $p<1$ are used in this way~\cite{DBLP:conf/aaai/NieHD12}.

Achlioptas and McSherry~\cite{achlioptas2007fast} initiated a long line of work
\cite{arora2006fast,gittens2009error,drineas2011note,nguyen2015tensor,achlioptas2013near,kundu2014note,kundu2017recovering,DBLP:conf/colt/BravermanK0S21}
that aims to find a sparse matrix $A'$ such that $\|A'-A\|\leq \epsilon \|A\|$ for $\epsilon>0$, where $\|\cdot\|$ is the spectral norm, which is also the largest singular value or the Schatten $\infty$-norm.
The main algorithmic technique in this line of work is to sample entries independently
according to a well-crafted distribution.
The error matrix $E=A'-A$ is then analyzed using tools from random matrix theory and is often of high rank.
An undesired byproduct is that if the original matrix $A$ is of low rank, $A'$ might be of full rank, and thus a bad approximation of $A$ in the Schatten $0$-norm.
Additionally, there is no guarantee that the Frobenius norm of $A'$ would be similar to that of $A$, and indeed there are simple counterexamples.
Perhaps surprisingly, little is known about sparsification in Schatten $p$-norm for other values of $p$, except for $p=2$ which is the Frobenius norm and reduces to vectors. 
Sparsification algorithms for other values of $p$ could open the door for new applications.

A recent line of work studies sparse and low-rank decomposition, where, given a matrix $A$, it is decomposed as $A = S + B$, such that $S$ is sparse and $B$ is low rank~\cite{CHANDRASEKARAN20091493}. 
In our language, this corresponds to approximating $A$ in Schatten $0$-norm by
considering the matrix $B$ as the error matrix.
The same paper~\cite{CHANDRASEKARAN20091493} further uses the Schatten $1$-norm as a surrogate for the rank, which then corresponds to sparsification in Schatten $1$-norm.
However, they differ from our approach by also relaxing the sparsity of $S$ to the $\ell_1$-norm.

Characterizing which matrices can be sparsified for each $p$ is a wide open problem (except for the easy case $p=2$); in fact not even resolved for $p=\infty$.
As a first step, we ask whether sparsification with respect to one norm, say Schatten $p$-norm, implies sparsification with respect to a different norm, say Schatten $q$-norm. 
An affirmative answer could be tremendously useful, because of the already large body of work on $p=\infty$ (and the easy case $p=2$).
We now define matrix approximation with respect to a general norm $\|\cdot\|_N$, although our results deal only with Schatten norms.

\begin{definition}\label{def:general_norm}\label{def:Schatten_norm} 
Given norm $\|\cdot\|_N:\R^m\to\R_+$ and accuracy parameter $\epsilon>0$, 
an \emph{$(\epsilon,N)$-norm approximation} of $x\in \R^m$
is any $x'\in \R^m$
such that $\|x-x'\|_N\leq \epsilon \|x\|_N$.
If $x'$ has at most $s>0$ non-zero entries, we call it an \emph{$(\epsilon,N,s)$-norm sparsifier} of $x$.
\end{definition}

We now instantiate this definition to the Schatten $p$-norm for $p\geq 0$.
An \emph{$(\epsilon,S_p,s)$-norm sparsifier} of a matrix $A\in\R^{n\times d}$ is a matrix $A'\in\R^{n\times d}$ that has at most $s>0$ non-zero entries and satisfies%
\footnote{We define it for the general case of rectangular matrices, but we focus on square matrices.}
\[
\|A-A'\|_{S_p}\leq \epsilon \|A\|_{S_p}.
\]

One may wonder if such a sparsification is even possible (except for trivial matrices).
However, the existence of spectral sparsification of graph Laplacians~\cite{DBLP:journals/siamcomp/SpielmanT11,BatsonSS12} implies that every graph Laplacian admits an $(\epsilon,S_p,O(\epsilon^{-2}n))$-norm sparsifier, in fact simultaneously for all $p\geq 1$, see Section~\ref{subsec:related_work}. 
With the above terminology at hand, the question raised earlier can be formulated as follows.
Throughout, we write $\tilde{O}(f)$ as a shorthand for $O(f\polylog n)$, and $\tilde{\Omega}(f)$ as a shorthand for $\Omega(f/\polylog n)$.

\begin{question}\label{question:Sp_to_Sq}
For $p\neq q$ and $0<\epsilon<1$, does an $(\epsilon,S_p,s)$-norm sparsifier for a matrix $A\in\R^{n\times d}$
necessarily imply (existentially and/or algorithmically)
also an $(\epsilon',S_q,s')$-norm sparsifier with $s'=\tilde{O}(s)$ \and $\epsilon'\leq \epsilon^{\Omega(1)}$?
\end{question}

As a first step, it is instructive to consider $(\epsilon,\ell_p)$-norm sparsifiers for vectors, since $\ell_p$ norm is analogous to the Schatten $p$-norm, but for vectors instead of matrices.
This can also be viewed as Question~\ref{question:Sp_to_Sq} in the special case of diagonal matrices.

\begin{question}\label{question:ellp_to_ellq}
For $p\neq q$ and $0<\epsilon<1$, does an $(\epsilon,\ell_p,s)$-norm sparsifier for a vector $x\in \R^n$
necessarily imply (existentially and/or algorithmically) also an $(\epsilon',\ell_q,s')$-norm sparsifier with $s'=\tilde{O}(s)$ and $\epsilon'\leq \epsilon^{\Omega(1)}$?
\end{question}


Our main finding is a surprising contrast in these two questions: 
a \emph{mostly affirmative answer} for Question~\ref{question:ellp_to_ellq}
(roughly for all $p<q$) with $\epsilon'=\epsilon$,
but a \emph{resounding negative answer} for Question~\ref{question:Sp_to_Sq}
(roughly for all $p\neq q$)
even when $\epsilon'$ is allowed to be a fixed constant.

An easy case is when $p$ and $q$ are sufficiently close,
and then the answer is affirmative, with $\epsilon'=2\epsilon$,
for both $\ell_p$ and Schatten $p$-norm, as follows.
For all $p<q$ and $a\in\R^n$,
H\"older's inequality implies that
$\|a\|_{q}\leq \|a\|_{p} \leq n^{\frac{1}{p}-\frac{1}{q}}\|a\|_{q}$. 
If $n^{\frac{1}{p}-\frac{1}{q}}$ is small, say at most $1+\epsilon$,
the two norms are approximately equal,
and thus sparsification in one norm immediately implies sparsification in the other norm.
The same holds for the Schatten $p$-norm as well.
The questions remain interesting when $n^{|\frac{1}{p}-\frac{1}{q}|}$ is not small,
particularly when $p\neq q$ are fixed and $n$ tends to infinity.

\subsection{Main Results}\label{sec:mainResults}

\paragraph{Vectors.}
For the family of $\ell_p$-norms,
we show that if $p<q$ then the answer to Question~\ref{question:ellp_to_ellq} is \emph{affirmative},
i.e., an $(\epsilon,\ell_p)$-norm sparsifier implies also an $(\epsilon,\ell_q)$-norm sparsifier using similar (but not identical) sparsity.
This is formalized in the next theorem, 
whose proof appears in Section~\ref{sec:vectors}.

\begin{restatable}{theorem}{ellpThm}\label{thm:ell_p_reduction}
Let $1\leq p < q$. Then
for all $0<\epsilon<\tfrac{1}{e}$ and $s>0$,
if $x\in\R^n$ has an $(\epsilon,\ell_p,s)$-norm sparsifier,
then it also has an $(\epsilon,\ell_q,O(s))$-norm sparsifier. 
\end{restatable}

The hidden constant in this $O(s)$ bound
is independent of $p,q$ and $\epsilon$,
and in fact, such a constant-factor loss in sparsity is necessary.
Consider, say, $p=1$ and $q=\infty$,
and a vector $x$ with $s+t$ non-zero coordinates,
$s$ of which equal $1$, and $t< s$ of which equal $\tfrac{s}{t}\epsilon$;
then $x$ has an $(\epsilon,\ell_1,s)$-norm sparsifier,
but every $(\epsilon,\ell_\infty)$-approximation requires $s+t$ non-zeros.
Fixing $t=\tfrac{s}{2}$, one can see that some loss in sparsity is indeed necessary.

In the other case of $p>q$, 
the answer to Question~\ref{question:ellp_to_ellq} is \emph{negative}
whenever $n^{\frac{1}{q}-\frac{1}{p}}\geq\epsilon^{-1}$.
Indeed, consider the vector $x=(1,\tfrac{1}{n^{1/q}},...,\tfrac{1}{n^{1/q}})$.
We denote by $e_1$ the first vector in the standard basis.
Without the first coordinate, $x$ satisfies $\|x-e_1\|_q \approx 1$ and $\|x-e_1\|_p = ((n-1)n^{-\frac{p}{q}})^{1/p} \approx n^{\frac{1}{p}-\frac{1}{q}} \leq \epsilon$, hence $e_1$ is a $(2\epsilon,\ell_p,1)$-norm sparsifier,
but not a $(0.1,\ell_q)$-norm approximation, and even taking more entries, say $x'=x_{\head(n/2)}$,
will not give a $(0.1,\ell_q)$-norm approximation.
Throughout, for integer $c<n$, we denote by $x_{\head(c)}$ the vector $x$ after zeroing out all but the $c$ largest entries in absolute value, breaking ties arbitrary.
Similarly, $x_{\tail(c)} = x - x_{\head(c)}$ is the vector $x$ after zeroing out the $c$ largest entries in absolute value.

\paragraph{Matrices.}
One may hope to extend the above result about $\ell_p$ norms of vectors
to Schatten norms of matrices.
Unfortunately, this is not possible, 
and in fact we answer Question~\ref{question:Sp_to_Sq} negatively for all fixed $p\neq q$, as follows.
Let $\epsilon_0>0$ be a sufficiently small fixed constant ($\epsilon_0 = 0.1$ works).%

\begin{restatable}{theorem}{schattenThm}\label{thm:main_forall_pq_separation}
Fix $p\neq q \geq 1$.
Then for all $n$ and $0.1 \geq\epsilon_0\geq \epsilon>(\log n)^{-|\frac{1}{q}-\frac{1}{p}|}$, there is a matrix $A\in\R^{n\times n}$ that has an $(\epsilon, S_p, O(n))$-norm sparsifier,
but every $(\epsilon_0, S_q)$-norm approximation of $A$ must have $\tilde{\Omega}(n^2)$ non-zero entries.
\end{restatable}

This theorem clearly needs some lower bound on $\epsilon$ to avoid the case $\epsilon=0$, yet our result covers the interesting case of fixed $\epsilon>0$ and $n\to \infty$.
We prove Theorem~\ref{thm:main_forall_pq_separation} in Section~\ref{sec:Schatten} by
providing explicit matrices $A\in\R^{n\times n}$, for infinitely many $n$.
We actually provide four families of matrices,
for different ranges of the parameters $p$ and $q$, some of which have a smaller lower bound on $\epsilon$.

These matrices may guide the design of sparsification algorithms for Schatten $p$-norms, e.g., for a specific value of $p$, one has to consider a class of matrices that do not include our hard instances for $p$.
  Thus, algorithms for one Schatten $p$-norm do not have to perform well for other Schatten $q$-norms, i.e., $q\neq p$, and might need to differ considerably.

We can further extend our results to many cases where $p<1$ or $q<1$,
including the important but exceptional case of the Schatten $0$-norm, 
as discussed next.
Our explicit constructions may be of independent interest.

\subsection{The Schatten $0$-norm}\label{subsec:intro_small_pq}
Perhaps the most important Schatten $p$-norm is the Schatten $0$-norm, which is the rank and is actually not a norm.
For this Schatten $0$-norm, Definition~\ref{def:Schatten_norm} instantiates to the following: given a matrix $A\in\R^{n\times d}$, decompose it into $A = S +R$ where $S$ is sparse and $\rank(R)\leq \epsilon\cdot\rank(A)$.
This is known from the optimization literature as sparse and low-rank decomposition~\cite{CHANDRASEKARAN20091493}, and from circuit complexity theory as Valiant's rigidity~\cite{DBLP:conf/mfcs/Valiant77}.
Hence, there is interest in the range $p<1$ (even though it is not a norm),
in part as a relaxation for the rank~\cite{DBLP:conf/aaai/NieHD12}.
Notice however that $\|A\|_{S_0}$ is not the limit of $\|A\|_{S_p}$ as $p\to 0$, 
but rather of $\|A\|_{S_p}^p$.
Thus, we must treat the Schatten $0$-norm separately,
rather than let $p\to 0$ in our results for $p>0$.

We show that Theorem~\ref{thm:main_forall_pq_separation} holds even for $0\leq p<1$ or $q<1$, except for $0<q<\min(p,1)$ where we do not have a proof.
See Section~\ref{sec:small_p_q} for formal statements and proofs.

\subsection{Related Work}\label{subsec:related_work}

\paragraph{Low-rank approximation}
Instead of sparsifying a matrix $A\in\R^{n\times d}$, a different goal is to decompose it as $A\approx UV$ where $U\in\R^{n\times k}$ and $V\in \R^{k\times d}$.
If $k$ is small, this is called a low-rank approximation.
Storing it only takes $O(k(n+d))$ space, and every multiplication by $A$ can be done by first applying $V$ and then applying $U$, which saves up on computation time.
One often measures the quality of this approximation in the
Frobenius norm,
and sometimes in other matrix norms.
Li and Woodruff~\cite{DBLP:conf/icml/LiW20} provide fast algorithms for low-rank approximation in the Schatten $p$-norm for all $p\geq 1$ by using dimension reduction for Ky-Fan $p$-norms.
Recently, Bakshi, Clarkson and Woodruff~\cite{DBLP:conf/stoc/BakshiCW22} used iterative Krylov methods to provide improved algorithms using a small number of matrix-vector products.

\paragraph{Sparsification in other matrix norms}
Gittens and Tropp~\cite{gittens2009error} considered sparsification in $\ell_\infty \to \ell_p$ operator norms, defined as $\|A\|_{\ell_\infty \to \ell_p} = \max_{x\neq 0} \frac{\|Ax\|_p}{\|x\|_\infty}$.
They focus on $\ell_\infty \to\ell_1$, which has strong equivalence to the cut norm, and on $\ell_\infty \to \ell_2$, which is relevant to applications such as column subset selection.

\paragraph{Spectral sparsification}
Graph Laplacians admit very good spectral sparsification, as follows.
A positive semi-definite matrix (PSD) $A'\in\R^{n\times n}$ is said to be an \emph{$\epsilon$-spectral approximation} of a PSD matrix $A\in\R^{n\times n}$ if $(1-\epsilon) A  \preceq A' \preceq (1+\epsilon) A$~\cite{DBLP:journals/siamcomp/SpielmanT11}.%
\footnote{For symmetric matrices $A,B$, we denote $A\preceq B$ if $B-A$ is PSD.}
This notion is stronger than Schatten $p$-norm approximation, 
by the following observation.

\begin{restatable}{lemma}{lemmaSpectralToSchatten}\label{lem:spectral_to_schatten}
For all PSD matrices $A\in\R^{n\times n}$ and $\epsilon>0$, 
every $\epsilon$-spectral approximation $A'$ of $A$ is also an $(\epsilon,S_p)$-norm approximation of $A$, simultaneously for all $p\geq 1$.
\end{restatable}

A proof of Lemma~\ref{lem:spectral_to_schatten} is provided in Appendix~\ref{appendix} for completeness.
For a graph Laplacian, one can compute an $\epsilon$-spectral sparsifier that has $O(\epsilon^{-2}n)$ non-zero entries~\cite{BatsonSS12}, see also~\cite{LeeSun18} for the latest time bounds for computing such sparsifiers.

\subsection{Notations}\label{subsec:notations}

For a matrix $A\in\R^{n\times d}$, we denote the $i$-th row by $A_i$ and the $j$-th column by $A^j$.
For sets $S\in [n]$ and $T\in [d]$, we denote by $A_{(S)}$ the submatrix made of rows of $A$ indexed by $S$, and by $A^{(T)}$ the submatrix made of columns of $A$ indexed by $T$.
Similarly, $A_{(S)}^{(T)}$ is the submatrix made of rows indexed by $S$ and columns indexed by $T$.
We denote the number of non-zero entries in $A$ by $\nnz(A)$.
We assume that $n$ is a power of $2$, which can be obtained by padding with $0$'s
and affects the derived bounds only by a constant factor.


%% file: EllSparsif.tex
\section{Simple Case: Vectors}\label{sec:vectors}
In this section we study the simpler case of $(\epsilon,\ell_p)$-norm sparsification.
We now restate and prove Theorem~\ref{thm:ell_p_reduction}.

\ellpThm*

As was discussed in Section~\ref{sec:mainResults}, this does not hold for fixed $p>q$.
Our proof of Theorem~\ref{thm:ell_p_reduction} will use the following generalization of Lemma 3 from~\cite{gupta2018exploiting}.
\begin{lemma}\label{lem:genera_NS_tail}
For all $a\in\R^n,1\leq p < q$ and integer $c<n$, we have $\|a_{\tail(c)}\|_q\leq c^{-(\frac{1}{p}-\frac{1}{q})}\|a\|_p$.
\end{lemma}
\begin{proof}[Proof of Lemma~\ref{lem:genera_NS_tail}]
Without loss of generality, assume that $a_1\geq a_2\geq ...\geq a_n\geq 0$. Note that $\|a_{\tail(c)}\|_\infty = a_{c+1}\leq \frac{\|a\|_p}{c^{1/p}}$, which proves the $q=\infty$ case. For finite $q$,
\[
\|a_{\tail(c)}\|_q^q = a_{c+1}^q +...+ a_n^q\leq a_{c+1}^{q-p}(a_{c+1}^p+...+a_n^p)
\leq a_{c+1}^{q-p}\|a\|_p^p \leq c^{-\frac{q-p}{p}}\|a\|_p^q.
\]
\end{proof}

\begin{proof}[Proof of Theorem~\ref{thm:ell_p_reduction}]
Denote by $x'$ an $(\epsilon,\ell_p,s)$-norm sparsifier of $x$.
Thus,
\[
\|x_{\tail(s)}\|_p^p \leq \|x'-x\|_p^p \leq \epsilon^p \|x\|_p^p = \epsilon^p (\|x_{\tail(s)}\|_p^p + \|x_{\head(s)}\|_p^p).
\]
Hence
\[
\|x_{\tail(s)}\|_p \leq \frac{\epsilon}{(1-\epsilon^p)^{1/p}}\|x_{\head(s)}\|_p\leq \frac{\epsilon}{(1-\epsilon^p)^{1/p}} s^{\frac{1}{p}-\frac{1}{q}} \|x_{\head(s)}\|_q.
\]
Take $c=\big(\frac{(1-\epsilon^q)^{1/q}}{(1-\epsilon^p)^{1/p}}\big)^{\frac{1}{\frac{1}{p}-\frac{1}{q}}}\cdot s$ more entries from $x$, i.e, let $x''=x_{\head(c+s)}$.
By Lemma~\ref{lem:genera_NS_tail},
\[
\|x_{\tail(c+s)}\|_q \leq c^{-(\frac{1}{p}-\frac{1}{q})} \|x_{\tail(s)}\|_p \leq 
\frac{\epsilon}{(1-\epsilon^p)^{1/p}}\big(\frac{s}{c}\big)^{\frac{1}{p}-\frac{1}{q}} \|x_{\head(s)}\|_q 
= \frac{\epsilon}{(1-\epsilon^q)^{1/q}}\|x_{\head(s)}\|_q.
\]
Hence $\|x_{\tail(c+s)}\|_q^q \leq \epsilon^q(\|x_{\tail(c+s)}\|_q^q + \|x_{\head(s)}\|_q^q) \leq \epsilon^q \|x\|_q^q$.
Next, we bound $c/s=O(1)$ by the following calculation. 
\begin{align*}
    \ln (c/s)
    &= {\tfrac{1}{\frac{1}{p}-\frac{1}{q}}}\big(\tfrac{1}{q} \ln (1-\epsilon^q)-\tfrac{1}{p} \ln (1-\epsilon^p)\big) \\
    &\leq_* {\tfrac{1}{\frac{1}{p}-\frac{1}{q}}}     \big( \tfrac{\epsilon^p}{p}  -  \tfrac{\epsilon^q}{q} \big) \\
    &= \epsilon^p + {\tfrac{1}{\frac{q}{p}-1}}     \big( \epsilon^p  -  \epsilon^q \big) \\
    &=\epsilon^p + {\tfrac{1}{\frac{q}{p}-1}}     \epsilon^p\big(1   -  e^{(q-p)\ln \epsilon} \big) \\
    &\leq \epsilon^p + {\tfrac{1}{\frac{q}{p}-1}}     \epsilon^p (q-p)\ln \tfrac{1}{\epsilon} \\
    &= \epsilon^p + p     \epsilon^p \ln \tfrac{1}{\epsilon} \\
    &= \epsilon^p (1 + p\ln \tfrac{1}{\epsilon}) \leq 1,
\end{align*}
where the last two inequalities use that $1+y\leq e^y$ for all $y$,
and the first inequality, marked $\leq_*$,
holds since $f(x) = \tfrac{1}{x}\ln(1-\epsilon^x) + \tfrac{\epsilon^x}{x}$ is a decreasing function for $x\geq 1$, which can be verified by considering its derivative,
\begin{align*}
    f'(x) &= -\tfrac{\epsilon^x}{x^2} + \tfrac{\epsilon^x}{x}\ln \epsilon +\tfrac{1}{x}\tfrac{1}{1-\epsilon^x}(-\epsilon^x \ln \epsilon) \\
    &\leq -\tfrac{\epsilon^x}{x^2} + \tfrac{2}{x}(\epsilon^{2x} \ln \tfrac{1}{\epsilon}) \\
    &= \tfrac{\epsilon^x}{x} (-\tfrac{1}{x} + 2\epsilon^{x} \ln \tfrac{1}{\epsilon}) \\
    &= \tfrac{\epsilon^x}{x^2} (-1 + 2e^{-x\ln \tfrac{1}{\epsilon}} x\ln \tfrac{1}{\epsilon}) <0,
\end{align*}
where the last step holds since $y e^{-y}\leq \tfrac{1}{e}$ and by substituting $y=x\ln \tfrac{1}{\epsilon}$.
In conclusion, $c=O(s)$ and $x''$ is an $(\epsilon,\ell_q,O(s))$-norm sparsifier of $x$.
\end{proof}


%% file: SchattenExamples.tex
\section{Schatten Norms}\label{sec:Schatten}

In this section we restate and prove Theorem~\ref{thm:main_forall_pq_separation}.
It shows that Schatten norms do not behave like $\ell_p$ norms, and even for $p<q$, the existence of an $(\epsilon,S_p)$-norm sparsifier does not mean existence of an $(\epsilon,S_q)$-norm sparsifier.

\schattenThm*

The proof is made of four cases, depending on whether $p>q$ (or $p<q$) and on whether $q\geq 2$ (or $1\leq q<2$).
Of these four cases,
perhaps the two most interesting ones are when $p<q$, because they stand in contrast to vectors (for vectors we know by Theorem~\ref{thm:ell_p_reduction} that an $(\epsilon,\ell_p)$-norm sparsifier implies an $(\epsilon,\ell_q)$-norm sparsifier with a similar number of non-zero entries).
While it would make sense to start with the proof of these two most interesting cases ($p<q$), due to some similarities in the proof technique, we organize our proof as described in Table~\ref{table:Schatten_examples}
(starting with the two cases where $q\geq 2$, and then the two cases where $q<2$).

Each of these four cases is proved by providing a matrix $A = A' + B$ that satisfies the following properties:
\begin{enumerate}[label={(\bfseries P\arabic*)}]
    \item $\nnz(A') = O(n)$.\label{P1}
    \item $\|B\|_{S_p} < \epsilon\|A'\|_{S_p}$.\label{P2}
    \item $\|A'\|_{S_q} = \|B\|_{S_q}$.\label{P3}
    \item Any $(2\epsilon_0, S_q)$-norm approximation of $B$ must have $\tilde{\Omega}(n^2)$ non-zero entries.\label{P4}
\end{enumerate}

\begin{lemma}\label{lem:properties}
For $p\neq q$ and $0<\epsilon<1/2$,
if matrices $A'$ and $B$ satisfy Properties~\ref{P1}-\ref{P4}, then $A'$ is an $(\epsilon,S_p)$-norm approximation of $A=A'+B$, and every $(\epsilon_0,S_q)$-norm approximation of $A$ must have $\tilde{\Omega}(n^2)$ non-zero entries,
for a fixed constant $\epsilon_0=0.1$.
\end{lemma}

\begin{proof}
By the triangle inequality and Property~\ref{P2}, 
$\|B\|_{S_p}< \epsilon \|A'\|_{S_p}\leq\epsilon(\|B\|_{S_p} + \|A\|_{S_p})$, 
hence
$\|A-A'\|_{S_p} = \|B\|_{S_p} < \epsilon /(1-\epsilon)\|A\|_{S_p}<2\epsilon \|A\|_{S_p}$ and thus by Property~\ref{P1},
$A'$ is a $(2\epsilon, S_p,O(n))$-norm sparsifier of $A$.
By Property~\ref{P3} and the triangle inequality, $\|A\|_{S_q}\leq 2\|B\|_{S_q}$.
Hence, any matrix $\tilde{A}$ that is an $(\epsilon_0, S_q)$-norm approximation of $A$, satisfies $\|A - \tilde{A}\|_{S_q} \leq \epsilon_0 \|A\|_{S_q}\leq 2\epsilon_0 \|B\|_{S_q}$.
Thus, $\tilde{A}-A'$ is a $(2\epsilon_0, S_q)$-norm approximation of $B$.
By Properties~\ref{P1} and \ref{P4}, it follows that $\nnz(\tilde{A}) \geq \tilde{\Omega}(n^2)$.
\end{proof}

\begin{table*}[!t]
\caption{\label{table:Schatten_examples} The four cases proved in Theorem~\ref{thm:main_forall_pq_separation} and the sections containing their proofs.}
\begin{center}
\begin{tabular}{ |c|c|c| }
\hline
        &  $1\leq p<q$                  &     $p>q$ \\
        \hline
 $q\geq 2$  & Section~\ref{sec:small_notto_large_p>2}    & Section~\ref{sec:p>q>2}\\
 \hline
 $1\leq q<2$  & Section~\ref{sec:p<q<2} & Section~\ref{sec:q<min_p_2}\\
 \hline
\end{tabular}
\end{center}
\end{table*}

\subsection{Case $q\geq 2$ and $1\leq p<q$}\label{sec:small_notto_large_p>2}

\begin{theorem}\label{thm:q>p>2}
Fix $q\geq 2$ and $1\leq p<q$. Then for all $n$ and $\epsilon>(\log n)^{-2(\frac{1}{p}-\frac{1}{q})}$,
there is a matrix $A\in\R^{n\times n}$ that has an $(\epsilon, S_p,n)$-norm sparsifier, but every $(\epsilon_0, S_q)$-norm approximation of $A$ must have $\Omega(n^2/\log^2 n)$ non-zero entries, 
for a fixed constant $\epsilon_0=0.1$.
\end{theorem}

Our proof is based on constructing matrices $A',B$ satisfying \ref{P1}-\ref{P4}. 
We set $A' = a I$ with parameter $a>0$ chosen specifically to satisfy \ref{P3}. 
We construct a matrix $B\in\R^{n\times n}$, having entries in $\{+1,-1\}$ and $m=n^{2(1-\alpha)}$ non-zero singular values all equal to $ n^\alpha \coloneqq \sqrt{n} \cdot \log n $ for a parameter $1>\alpha>1/2$. 
Note that $\|B\|_{S_2}^2 = n^{2(1-\alpha)}(n^\alpha)^2 = n^2$.
The construction is as follows: 
let $v$ be the all-ones column vector of dimension $n^{2\alpha-1}$ and $H_{[m]}$ the matrix made of the first $m$ rows of the $n\times n$ Hadamard matrix.%
\footnote{Hadamard matrices are not known for every $n$, but are known for powers of $2$.
Recall that we assumed $n$ is a power of $2$ in Subsection \ref{subsec:notations}, thus there exists an $n\times n$ Hadamard matrix.}
Then, construct $B= H_{[m]} \otimes v$, where $\otimes$ denotes the Kronecker product,
i.e, all the rows of $B_{U_i}$ are equal to the $i$-th row of $H$, where  $U_i = [(i-1) n^{2\alpha-1} + 1,i\cdot n^{2\alpha-1}]$ for $i\in[1,m]$,
and the number of rows in $B$ is $n^{2\alpha-1}m = n$.
Since the vector $v$ has a single singular value which is just its $\ell_2$-norm and $H_{[m]}$ has $m$ singular values of equal value $\sqrt{n}$, by properties of Kronecker product, $B$ has $m$ singular values all equal to $\sqrt{n^{2\alpha-1}}\sqrt{n} = n^\alpha$, as desired.

We use the following lemmas to show that this matrix $B$ satisfies \ref{P4}.

\begin{lemma}\label{lem:rank1_spectral_LB} 
Every matrix $(B')_{U_i}$ that is an $(\epsilon,S_\infty)$-norm approximation of $B_{U_i}$, must have $\nnz((B')_{U_i})\geq \tfrac{n}{2}$, even for $\epsilon=\tfrac{1}{2}$.
\end{lemma}
\begin{proof}

Assume by contradiction that $(B')_{U_i}$ is a $(\tfrac{1}{2},S_\infty)$-norm approximation of $B_{U_i}$ and that $\nnz((B')_{U_i}) < \tfrac{n}{2}$.
Then, there is a set of columns $Z$ of size at least $\tfrac{n}{2}$, such that 
$(B')_{U_i}^{Z}=0$.
Hence, 
\begin{align*}
\|(B'-B)_{U_i}\|_{S_\infty}
&=\max_{\|x\|_2=1} \|(B'-B)_{U_i}x\|_2
\geq \|(B'-B)_{U_i}^Z\|_{S_\infty}\\
&= \|B_{U_i}^Z\|_{S_\infty}
\geq \sqrt{\tfrac{n}{2}  \cdot  n^{2\alpha-1}}
= \tfrac{n^\alpha}{\sqrt{2}} = \tfrac{\|B_{U_i}\|_{S_\infty}}{\sqrt{2}},
\end{align*}
a contradiction.
Hence every $(\tfrac{1}{2},S_\infty)$-norm approximation of $B_{U_i}$ must have at least $\tfrac{n}{2}$ non-zero entries.
\end{proof}

Our proof builds on the Pinching inequality,
which we state first for completeness. 

\begin{lemma}[Pinching inequality~\cite{DBLP:journals/tamm/Bhatia00} (see also \cite{bhatia2002pinchings})]\label{lem:pinching_bhatia}
For all $A\in\R^{n\times n}$, $p\geq 1$ and a collection of $k$ disjoint subsets $I_j\subset [n]$ for $j = 1,...,k$,
\[
\|A\|_{S_p}^p \geq \sum_{j=1}^k \|A_{I_j}^{I_j}\|_{S_p}^p.
\]
\end{lemma}

\begin{lemma}\label{lem:pinching_block_diag}
For $q\geq 2$, a matrix $A\in\R^{n\times d}$ and a disjoint collection of sets $U_i\subset [d]$, it holds that $\|A\|_{S_q}^q\geq \sum_i \|A^{(U_i)}\|_{S_\infty}^q$.
\end{lemma}

\begin{proof}

By Lemma~\ref{lem:pinching_bhatia} and since $\| X\|_{S_{q/2}}
\geq \| X\|_{S_{\infty}}$ for every matrix $X$,
\[
\|A\|_{S_q}^q 
= \|A^\top A\|_{S_{q/2}}^{q/2} 
\geq \sum_i\| (A^\top A)_{U_i}^{U_i}\|_{S_{q/2}}^{q/2} 
\geq \sum_i\| (A^\top A)_{U_i}^{U_i}\|_{S_{\infty}}^{q/2} 
= \sum_i \|A^{(U^i)}\|_{S_{\infty}}^q.
\]
\end{proof}

\begin{lemma}\label{lem:B_hard_for_p>2} 
For all $q\geq 2$, 
every $(\tfrac{1}{4},S_q)$-norm approximation of the matrix $B$
must have $\Omega(n m)$ non-zero entries.
\end{lemma}
\begin{proof}
Note that $\|B\|_{S_q} = m^{1/q} n^\alpha$.
Consider $B'$ that has $s_i$ non-zero entries in the rows indexed by $U_i$.
By Lemma~\ref{lem:pinching_block_diag}, 
\[
\|B'-B\|_{S_q}^q
\geq \sum_i \|(B'-B)_{U_i}\|_{S_\infty}^q.
\]
Even if $\nnz(B')=\tfrac{n m}{4}$, there are at least $\tfrac{m}{2}$ sets $U_i$ where $s_i\leq \tfrac{n}{2}$. 
Hence, for these sets, by Lemma~\ref{lem:rank1_spectral_LB},
$\|(B'-B)_{U_i}\|_{S_\infty} > \tfrac{1}{2} \|(B)_{U_i}\|_{S_\infty} = \tfrac{1}{2} n^\alpha$,
thus the sum above is at least 
$\tfrac{m}{2} (\tfrac{1}{2})^q  n^{\alpha q}=(\tfrac{1}{2})^{q+1} \|B\|_{S_q}^q$, which concludes the proof of the lemma.
\end{proof}

\begin{proof}[Proof of Theorem~\ref{thm:q>p>2}]
Since $n^\alpha = \sqrt{n} \log n$, then 
$m=\tfrac{n}{\log^2 n}$ and $\|B\|_{S_p} = m^\frac{1}{p}n^\alpha =n^{\frac{1}{2}+\frac{1}{p}}(\log n)^{1-\frac{2}{p}}$.
The statement of the theorem holds when $\epsilon > (\log n)^{-2(\frac{1}{p}-\frac{1}{q})}$.

Let $A=\sqrt{n}(\log n)^{1-\frac{2}{q}} I + B$, i.e, set $A'\coloneqq \sqrt{n}(\log n)^{1-\frac{2}{q}} I$.
We now verify
\begin{enumerate}[label={\bfseries P\arabic*}]
    \item $\nnz(A') = n$.
    \item $\|\tfrac{1}{\sqrt{n}(\log n)^{1-\frac{2}{q}}}B\|_{S_p} = n^{\frac{1}{p}} (\log n)^{-2(\frac{1}{p}-\frac{1}{q})} < \epsilon \|I\|_{S_p}$.
    \item $\|\tfrac{1}{\sqrt{n}(\log n)^{1-\frac{2}{q}}}B\|_{S_q} = n^{\frac{1}{q}} = \|I\|_{S_q}$.
    \item follows from Lemma~\ref{lem:B_hard_for_p>2}.
\end{enumerate}
The matrices $A'$ and $B$ satisfy \ref{P1}-\ref{P4}, and by Lemma~\ref{lem:properties}, this concludes the proof of Theorem~\ref{thm:q>p>2}.
\end{proof}

\subsection{Case $p>q\geq 2$}\label{sec:p>q>2}

\begin{theorem}\label{thm:p>q>2}
Fix $p> q\geq 2$. Then for all $n$ and $\epsilon>n^{-(\frac{1}{q}-\frac{1}{p})}$,
there is a matrix $A\in\R^{n\times n}$ that has an $(\epsilon, S_p,1)$-norm sparsifier, but every $(\epsilon_0, S_q)$-norm approximation of $A$ must have $\Omega(n^2)$ non-zero entries,
for a fixed constant $\epsilon_0=0.1$.
\end{theorem}

Again, the proof is based on constructing matrices $A',B$ satisfying \ref{P1}-\ref{P4}. 
We set $A'= J^1$, a matrix having a single non-zero entry of value $1$ 
located at entry $(1,1)$;
and set $B = a H$, a Hadamard matrix scaled by suitable $a>0$ to satisfy \ref{P3}.

We get the following from Lemma~\ref{lem:B_hard_for_p>2}, by setting $\alpha=\tfrac{1}{2}$ and $m=n$.

\begin{lemma}\label{lem:hard_for_p>2}
For all $q\geq 2$, 
every $(\frac{1}{4},S_q)$-norm approximation of an $n\times n$ Hadamard matrix $H$
must have $\Omega(n^2)$ non-zero entries.
\end{lemma}

\begin{proof}[Proof of Theorem~\ref{thm:p>q>2}]
Set $a = n^{-\frac{1}{q}-\frac{1}{2}}$, i.e,
set $A'\coloneqq J^1$ and $B\coloneqq n^{-\frac{1}{q}-\frac{1}{2}} H$.
We now verify
\begin{enumerate}[label={\bfseries P\arabic*}]
    \item $\nnz(J^1) = 1$.
    \item $\|n^{-\frac{1}{q}-\frac{1}{2}} H\|_{S_p} = n^{\frac{1}{p}-\frac{1}{q}} \leq \epsilon = \epsilon \|J^1\|_{S_p}$.
    \item $\|J^1\|_{S_q} = 1 = \|n^{-\frac{1}{q}-\frac{1}{2}} H\|_{S_q}$.
    \item follows from Lemma~\ref{lem:hard_for_p>2}.
\end{enumerate}
The matrices $A'$ and $B$ satisfy \ref{P1}-\ref{P4}, and by Lemma~\ref{lem:properties}, this concludes the proof of Theorem~\ref{thm:p>q>2}.
\end{proof}

\subsection{Case $1\leq p < q\leq 2$}\label{sec:p<q<2}

\begin{theorem}\label{thm:p<q<2}
Fix $1\leq p < q\leq 2$. Then for all $n$ and $\epsilon>n^{-(\frac{1}{p} - \frac{1}{q})}$, 
there is a matrix $A\in\R^{n\times n}$ that has an $(\epsilon, S_p,n)$-norm sparsifier, but every $(\epsilon_0, S_q)$-norm approximation of $A$ must have $\Omega(n^2)$ non-zero entries,
for a fixed constant $\epsilon_0=0.1$.
\end{theorem}
Again, our proof is based on constructing matrices $A',B$ satisfying \ref{P1}-\ref{P4}.
Let $J_n$ denote the $n\times n$ all-ones matrix.
We set $A'= I$ and $B = a J_n$, with scalar $a>0$ chosen specifically to satisfy \ref{P3}.

\begin{lemma}\label{lem:J_hard_for_p<2}
For all $0< q\leq 2$, every $(\epsilon_0,S_q)$-norm approximation of $J_n$ must have $\Omega(n^2)$ non-zero entries.
\end{lemma}
\begin{proof}
$J_n$ is rank-1, hence $\|J_n\|_{S_2}=\|J_n\|_{S_q}=n$.
It is clear that every Frobenius norm approximation $J'$ of $J_n$ must have at least $\Omega(n^2)$ non-zero entries.
Since $\|J'-J_n\|_{S_q}\geq \|J'-J_n\|_{S_2}$, the conclusion follows.
\end{proof}

\begin{proof}[Proof of Theorem~\ref{thm:p<q<2}]
Let $A=I + n^{\frac{1}{q}-1} J_n$
and set $A'\coloneqq I$ and $B\coloneqq n^{\frac{1}{q}-1} J_n$.
We now verify
\begin{enumerate}[label={\bfseries P\arabic*}]
    \item $\nnz(I) = n$.
    \item $\|n^{\frac{1}{q}-1} J_n\|_{S_p} = n^{\frac{1}{q}} = n^{\frac{1}{q} - \frac{1}{p}}\|I\|_{S_p} < \epsilon \|I\|_{S_p}$.
    \item $\|n^{\frac{1}{q}-1} J_n\|_{S_q} = n^{\frac{1}{q}} = \|I\|_{S_q}$.
    \item follows from Lemma~\ref{lem:J_hard_for_p<2}.
\end{enumerate}
By Lemma~\ref{lem:properties}, this concludes the proof of Theorem~\ref{thm:p<q<2}.
\end{proof}

\subsection{Case $1\leq q<\min(p,2)$}\label{sec:q<min_p_2}
\begin{theorem}\label{thm:q<min_p_2}
Fix $1\leq q<\min(p,2)$.
Then for all $n$ and $\epsilon > (\log n)^{-(\frac{1}{q}-\frac{1}{p})}$, 
there is a matrix $A\in\R^{n\times n}$ that has an $(\epsilon, S_p,1)$-norm sparsifier, but every $(\epsilon_0, S_q)$-norm sparsifier of $A$ must have $\Omega(n^2/\log n)$ non-zero entries,
for a fixed constant $\epsilon_0=0.1$.
\end{theorem}

Again, our proof is based on constructing matrices $A',B$ satisfying \ref{P1}-\ref{P4}.
We set $A'= J^1$ and $B = a C$, with scalar $a>0$ chosen specifically to satisfy \ref{P3}, and $C$ a block-diagonal matrix described below.

\begin{proof}[Proof of Theorem~\ref{thm:q<min_p_2}.]
Let $C$ be an $n\times n$ block-diagonal matrix, whose blocks are the matrices $\tfrac{1}{m}J_m$ with parameter $m$ to be specified.
It has $n/m$ singular values, each one equals $\tfrac{1}{m}\|J_m\| = 1$, and $\nnz(C) = n m$.
For every matrix $C'$, it holds by 
the pinching inequality~\cite{DBLP:journals/tamm/Bhatia00}
(Lemma~\ref{lem:pinching_bhatia})
that $\|C'-C\|_{S_q} \geq \|\blockdiag(C'-C)\|_{S_q}$, where $\blockdiag(C'-C)$ denotes the block-diagonal matrix that is $0$ in all locations where $C$ is zero, and otherwise has the same value as $C'-C$
(i.e, it has the same structure as $C$).
Hence, in order for $C'$ to be a $(\tfrac{1}{4},S_q)$-norm approximation, $\Omega(n/m)$ of the blocks on its diagonal have to be at least $(\tfrac{1}{2},S_q)$-norm approximation of $\tfrac{1}{m}J_m$, thus by Lemma~\ref{lem:J_hard_for_p<2}, $\nnz(C')\geq \Omega (n m)$.

Set $m=\tfrac{n}{\log n}$,
let $A = J^1 + (\log n)^{-\frac{1}{q}} C$
and set $A'\coloneqq J^1$ and $B\coloneqq (\log n)^{-\frac{1}{q}} C$.
We now verify
\begin{enumerate}[label={\bfseries P\arabic*}]
    \item $\nnz(J^1) = 1$.
    \item $\|(\log n)^{-\frac{1}{q}} C\|_{S_p} = (\log n)^{\frac{1}{p}-\frac{1}{q}} < \epsilon \|J^1\|_{S_p}$.
    \item $\|(\log n)^{-\frac{1}{q}} C\|_{S_q} = 1 = \|J^1\|_{S_q}$.
    \item was verified in the preceding paragraph.
\end{enumerate}
By Lemma~\ref{lem:properties}, this concludes the proof of Theorem~\ref{thm:q<min_p_2}.
\end{proof}


%% file: SmallPQ.tex
\section{When $p<1$ or $q<1$}\label{sec:small_p_q}

When $0<p<1$, the Schatten $p$-norm does not satisfy the triangle inequality and is thus not a norm,
but it nevertheless seems natural to extend the results of Theorem~\ref{thm:main_forall_pq_separation} to $p<1$ or $q<1$.
Unfortunately, our proof for each case in Theorem~\ref{thm:main_forall_pq_separation} constructs matrices $A'$ and $B$ satisfying \ref{P1}-\ref{P4} and then employs the triangle inequality.
However, we can replace the triangle inequality with the next lemma, incurring an inflation $1+\epsilon\to (1+\epsilon^p)^{1/p}$ when using \ref{P2} and deflation $2\to 2^{1/q}$ when using \ref{P3}.
As mentioned in Section~\ref{subsec:intro_small_pq}, we treat $p= 0$ and $q=0$ separately.

\begin{lemma}\label{lem:triangle_small_p}
For all $A,B\in\R^{m\times n}$ and $0<p\leq 1$, we have $\|A+B\|_{S_p}^p\leq \|A\|_{S_p}^p + \|B\|_{S_p}^p$.
\end{lemma}

This lemma is a special case of~\cite[Theorem 1]{rotfeld1967remarks} (see also \cite{thompson1976convex} and~\cite[Theorem IV.2.14]{bhatia1997matrix}), instantiated for the function $f(x)=x^p$ for $p<1$. 
The general statement applies to concave increasing functions $f:\R_+\to\R_+$ with $f(0)=0$~\cite[Theorem 1]{rotfeld1967remarks}.

We now use this lemma
to prove the following claim:
For all $0<p,q<1$ and $0<\epsilon < (2^p-1)^{1/p}$,
if $A= A'+B$ satisfies \ref{P1}-\ref{P4},
then $A'$ is a $(2\epsilon, S_p, O(n))$-norm sparsifier of $A$
and every $(2^{-1/q} \epsilon_0, S_q)$-norm approximation of $A$
must have $\tilde{\Omega}(n^2)$ non-zero entries. 
Let us now verify this claim. 
By~\ref{P2} and Lemma~\ref{lem:triangle_small_p}, 
$\|A-A'\|_{S_p} = \|B\|_{S_p} < \epsilon (1+\epsilon^p)^{1/p}\|A\|_{S_p}\leq 2\epsilon \|A\|_{S_p}$,
and thus by \ref{P1},
$A'$ is a $(2\epsilon, S_p, O(n))$-norm sparsifier of $A$.
By \ref{P3} and Lemma~\ref{lem:triangle_small_p}, $\|A\|_{S_q}\leq 2^{1/q}\|B\|_{S_q}$.
Hence, any matrix $\tilde{A}$ that is a $(2^{-1/q} \epsilon_0, S_q)$-norm approximation of $A$, satisfies $\|A - \tilde{A}\|_{S_q} \leq 2^{-1/q} \epsilon_0 \|A\|_{S_q}\leq \epsilon_0 \|B\|_{S_q}$, meaning
that $\tilde{A}-A'$ is an $(\epsilon_0, S_q)$-norm approximation of $B$.
By \ref{P1} and \ref{P4} it follows that $\nnz(\tilde{A}) \geq \tilde{\Omega}(n^2)$.

Similarly to Section~\ref{sec:Schatten}, we treat $p<q$ and $p>q$ separately, and in the case $p<q$ we treat subcases $q>2$ and $q<2$ separately.

\subsection{Case $0 \leq p<1$ and $q\geq 2$}
We use the same construction as in Section~\ref{sec:small_notto_large_p>2}.
For $p>0$, the same proof of \ref{P1}-\ref{P4} works, as most of the effort in Section~\ref{sec:small_notto_large_p>2} was in proving that the matrix $B$ satisfies \ref{P4}, which does not depend on $p$ and hence still holds.
For $p=0$, 
note that $A$ is of full rank ($\|A\|_{S_0}=n$), and $B$ is of rank $m = \tfrac{n}{\log^2 n}$.
Thus $A'$ is a $(\log^{-2} n, S_0, n)$-norm sparsifier of $A$.
We have thus proved the following.

\begin{corollary}
Fix $q\geq 2$ and $0\leq p<1$. Then for all $n$ and $\tfrac{1}{\log n} < \epsilon< (2^p-1)^{1/p}$, 
there is a matrix $A\in\R^{n\times n}$ that has an $(\epsilon, S_p, n)$-norm sparsifier, but every $(\epsilon_0, S_q)$-norm approximation of $A$ must have $\tilde{\Omega}(n^2)$ non-zero entries.
\end{corollary}

\subsection{Case $0\leq p<q<2$}
We use the same construction as in Section~\ref{sec:p<q<2}.
For $p>0$, the same proof of \ref{P1}-\ref{P3} works immediately, and \ref{P4} follows since Lemma~\ref{lem:J_hard_for_p<2} applies for $0<q<1$.
For $p=0$, it is easy to see that $A'$ is a $(\tfrac{1}{n},S_0,n)$-norm sparsifier of $A$ since $J_n$ is of rank $1$ and $A$ is of full rank, resulting with even better bounds.

\begin{corollary}
Fix $0 \leq p <1$ and $p< q\leq 2$. Then for all $n$ and $n^{-(\frac{1}{p} - \frac{1}{q})} <\epsilon< (2^p-1)^{1/p}$, there is a matrix $A\in\R^{n\times n}$ that has an $(\epsilon, S_p,n)$-norm sparsifier, but every $(\epsilon_0(q), S_q)$-norm approximation of $A$ must have $\Omega(n^2)$ non-zero entries,
where $\epsilon_0(q) = \epsilon_0$ if $q\geq 1$ and $2^{-1/q}\epsilon_0$ if $q<1$.
\end{corollary}

\subsection{Case $0\leq q<\min(1,p)$}
One could hope to use the same construction as in Section~\ref{sec:q<min_p_2}, but the proof of \ref{P4} no longer works, since it relies on the pinching inequality of~\cite{DBLP:journals/tamm/Bhatia00}, which in turn relies on the fact that the Schatten $p$-norm is indeed a norm. 
We do not have a proof for all $q$ in this range.
For $q=0$, there is a known
non-explicit hard instance, as follows.

\begin{lemma}[Proposition 3.3 of~\cite{PR94}]
There are constants $\epsilon_1,\epsilon_2>0$ and a matrix $B\in\{0,1\}^{n\times n}$, such that in order to reduce its rank to $\epsilon_1 n$, one must change at least $\epsilon_2 n^2$ entries.
\end{lemma}

For all $p>0$, let $A' = \tfrac{n}{\epsilon} I$.
Since all the singular values of $B$ are bounded by $n$, then
all the singular values of $A = A' + B$ are in $(\tfrac{1}{\epsilon}\pm1)n$.
Thus, $A'$ is a $(2\epsilon,S_p,n)$-sparsifier of $A$,
and by \cite{PR94}, every $(\epsilon_1,S_0)$-approximation of $A$ must have $\Omega(n^2)$ non-zero entries.

\begin{corollary}
Let $p>0$. Then for all $n$ and $\epsilon>1/\poly(n)$,
there is a matrix $A\in\R^{n\times n}$ that has a $(2\epsilon, S_p,n)$-norm sparsifier, but every $(\epsilon_1, S_0)$-norm approximation of $A$ must have $\Omega(n^2)$ non-zero entries.
\end{corollary}


%% file: appendix.tex
\appendix

\section{Proof of Lemma~\ref{lem:spectral_to_schatten}}\label{appendix}

In this section, we prove Lemma~\ref{lem:spectral_to_schatten}.

\lemmaSpectralToSchatten*

\begin{proof}
Let $A'\in\R^{n\times n}$ be an $\epsilon$-spectral approximation of $A$, i.e., $-\epsilon A \preceq A'-A \preceq \epsilon A$.
Observe that the matrix $A'-A$ is symmetric.
Let the eigendecomposition of $A'-A$ be $UDU^\top$, including zero eigenvalues so that $U\in\R^{n\times n}$ is unitary.
Denote the $i$-th column of $U$ by $u_i$ (which is a normalized eigenvector).
Then
\begin{align*}
  \|A'-A\|_{S_p}^p 
   &= \sum_i |u_i^\top (A'-A)u_i|^p
    \le \sum_i |u_i^\top (\epsilon A)u_i|^p 
    = \epsilon^p \sum_i (u_i^\top A u_i)^p 
\\
   &= \epsilon^p \|\diag(U^\top A U)\|_{S_p}^p
   \leq \epsilon^p \| U^\top A U \|_{S_p}^p
   = \epsilon^p \|A\|_{S_p}^p,
\end{align*}
where $\diag(U^\top A U)$ is a diagonal matrix with the same diagonal as $U^\top A U$ (and zeros otherwise),
and the last inequality holds by Lemma~\ref{lem:pinching_bhatia} (pinching inequality \cite{DBLP:journals/tamm/Bhatia00}).
\end{proof}
